\documentclass[reqno,a4paper,11pt]{amsart}

\usepackage[a4paper, margin=1in]{geometry}
\usepackage{colonequals}
\usepackage{microtype}
\usepackage[dvipsnames]{xcolor}
\usepackage{amssymb}
\usepackage{mathtools}
\usepackage{commath}
\usepackage{amsthm}
\usepackage{graphicx}

\theoremstyle{plain}
\newtheorem{theorem}{Theorem}[section]
\newtheorem{lemma}[theorem]{Lemma}
\newtheorem{proposition}[theorem]{Proposition}

\theoremstyle{definition}

\newtheorem*{note}{Note}

\theoremstyle{remark}
\newtheorem{example}{Example}

\newcommand{\C}{\mathbb{C}}
\newcommand{\Z}{\mathbb{Z}}
\newcommand{\Q}{\mathbb{Q}}
\newcommand{\N}{\mathbb{N}}
\newcommand{\R}{\mathbb{R}}
\newcommand{\E}{\mathbb{E}}
\newcommand{\Qbar}{\overline{\Q}}

\usepackage[pdfborderstyle={/S/U/W 0},unicode]{hyperref}
\hypersetup{
    colorlinks=true,
    linkcolor=blue!85!black,
    citecolor=blue!85!black,
    urlcolor=blue!85!black,
}

\numberwithin{equation}{section}

\usepackage{silence}




\renewcommand{\le}{\leqslant}
\renewcommand{\ge}{\geqslant}

\begin{document}

\title[Congruent copies of finite patterns in planar point sets]{Congruent copies of finite patterns in planar point sets}

\author[S. Bhattacharya]{Shubhrajit Bhattacharya}
\address{Department of Mathematics, The University of Chicago, Chicago, Illinois, USA}
\email{shubhrajit@uchicago.edu}

\author[R. Goenka]{Ritesh Goenka}
\address{Mathematical Institute, University of Oxford, Oxford, United Kingdom}
\email{ritesh.goenka@maths.ox.ac.uk}

\begin{abstract}
    Given a finite nonempty planar point set $S$, what is the maximum number of congruent copies of $S$ contained in a set of $n$ points in the Euclidean plane? Building on OpenAI's recent breakthrough on the unit distance problem, we construct planar sets consisting of $n$ points that contain $\Omega_S(n^{1+\delta_S})$ congruent copies of $S$, for some positive constant $\delta_S$ depending only on $S$. This answers a question of Brass and Pach in a strong form, and makes progress on questions posed by Erd\H{o}s and Purdy, and \'Abrego and Fern\'andez-Merchant. Our proof uses the number field construction from Sawin's quantitative refinement of OpenAI's result and consequently yields an explicit choice for $\delta_S$ for each fixed $S$.
\end{abstract}


\maketitle

\section{Introduction}
\label{sec:intro}

A wide range of problems in discrete geometry can be cast as problems about occurrences of certain geometric patterns in finite point sets. As in \cite{BP}, a \emph{geometric pattern} refers to an equivalence class of point sets in the $d$-dimensional Euclidean space under some fixed geometrically defined equivalence relation. Some common geometrically defined equivalence relations include congruence, similarity, and equivalence under translations. A natural question is to determine the maximum number of occurrences of a given pattern in an $n$-point set. In this paper, we restrict our attention to planar point sets. While this question is quite well understood for similarity and translation classes (see~\cite{BP}), the case of congruence classes remains far less understood.

Among the earliest examples of such a problem is the classical \emph{unit distance problem}, asked by Erd\H{o}s in 1946~\cite{Erd}. Here, the geometric pattern is simply a pair of points at distance one. Let $u(n)$ denote the maximum number of unordered pairs of points at distance one contained in any set of $n$ points in the plane. Erd\H{o}s~\cite{Erd} showed that $u(n) = \Omega(n^{1+\delta/\log \log n})$ for some constant $\delta > 0$ and $u(n) = O(n^{3/2})$. The best known upper bound $u(n) = O(n^{4/3})$ was proved by Spencer, Szemer\'edi, and Trotter in 1984~\cite{SST}. On the other hand, Erd\H{o}s's lower bound, which was commonly believed to be the true order of growth of $u(n)$, stood until the recent breakthrough by OpenAI~\cite{OpenAI}, where the authors improved it by a polynomial factor in $n$. Shortly afterwards, Sawin~\cite{Saw} gave a quantitative refinement of the construction, with an explicit exponent. Namely, he showed that there exists a positive constant $C$ such that $u(n) \ge C n^{1+\delta_0}$ for infinitely many $n$, where $\delta_0 = 0.014114$. We use this constant throughout the rest of this paper.

Moving to more general patterns, given a finite nonempty set $S \subset \R^2$, we define $f_S(n)$ as the maximum number of congruent copies of $S$ contained in any set of $n$ points in the plane. Trivially, $f_S(n) = n$ if $S$ is a singleton. The unit distance problem corresponds to the case where $S$ is any set of two points. Note that $f_{S'}(n) = f_S(n)$ if $S'$ is similar to $S$. Thus, when $|S| = 2$, the exact distance between the points is immaterial, and $f_S(n)$ simplifies to $u(n)$. Taking $\lfloor n/|S| \rfloor$ disjoint copies of $S$ and adding extra points as required yields $f_S(n) \ge \lfloor n/|S| \rfloor = \Omega_S(n)$. Conversely, $f_S(n) \le 4 u(n) = O(n^{4/3})$, because any congruent image of $S$ contained in the $n$-point set is already determined (up to reflection) by the mapping of two fixed points in $S$~\cite{BP}. In the following theorem, we extend the new lower bound for the unit distance problem to an arbitrary set $S$.

\begin{theorem}
\label{thm:main}
    Let $S \subset \R^2$ be a finite set with $|S| \ge 2$. Then there exist positive constants $C_S$ and $\delta_S$ depending only on $S$ such that $f_S(n) \ge C_S n^{1+\delta_S}$ for infinitely many $n$. Moreover, we have $f_S(n) \ge (C_S/2) n^{1+\delta_S/2}$ for all $n$ large enough in terms of $S$, which implies $f_S(n) = \Omega_S(n^{1+\delta_S/2})$. For sets $S$ that are affinely independent\footnote{A finite set $S \subset \C$ is said to be affinely independent over $\Qbar$ if upon translating an arbitrary point in $S$ to the origin, the remaining points are linearly independent over $\Qbar$.} over $\Qbar$, one may choose $\delta_S = \delta/(|S| 
    \log |S|)$ for some absolute constant $\delta > 0$.
\end{theorem}

The proof of Theorem~\ref{thm:main} can be used to obtain an explicit choice of $\delta_S$ for any given set $S$. Consequently, a general dependence of $\delta_S$ on $S$ can be extracted from the proof. However, we do not state it here since it is not easily expressible as an invariant of $S$ under similarity transformations. We compute the constant $\delta_S$ for a few different choices of $S$ in Section~\ref{sec:examples}. One example of particular interest is $S = \Delta$, where $\Delta$ denotes the vertex set of an equilateral triangle with unit side length. The problem of determining $f_\Delta(n)$ was raised in \cite{AF}. As shown in Section~\ref{sec:examples} (see Example~\ref{ex:cyclotomic}), Theorem~\ref{thm:main} implies $f_\Delta(n) = \Omega(n^{1+\delta_0/2})$, which improves upon the previously best known lower bound $\Omega(n^{1+\delta/\log \log n})$ for some constant $\delta > 0$.

Erd\H{o}s and Purdy~\cite{EP1, EP2} raised and studied the following question: determine the maximum number of pairwise congruent triangles spanned by $n$ points in the plane. This problem has since been reiterated in several places (see~\cite{BP}, \cite[Section~6.1]{BMP}, \cite{Pach}). Let $\mathcal{T}$ denote the collection of point sets $S$ consisting of points corresponding to the vertices of a triangle. Then Erd\H{o}s and Purdy's question is equivalent to determining $f(n) \colonequals \max_{S \in \mathcal{T}} f_{S}(n)$. It is known that $f(n) = \Omega(n^{1+\delta/\log \log n})$ for some constant $\delta > 0$ and $f(n) = O(n^{4/3})$ (see~\cite[Section~6.1]{BMP}). Theorem~\ref{thm:main} implies $f(n) \ge f_\Delta(n) = \Omega(n^{1+\delta_0/2})$, thus improving upon the above lower bound.

Brass and Pach asked the following problem\footnote{We state a slightly relaxed version of their problem since the original version asks for at least $n e^{c(S) \log n/\log \log n}$ congruent copies of $S$ for all $n$, which is not possible since $n$ must be at least $|S|$ to contain even one copy of $S$.}~\cite[Problem~6]{BP}: Does there exist, for every finite set $S$, a positive constant $c(S)$ such that $f_S(n) \ge n e^{c(S) \log n/ \log \log n}$ for $n$ large enough in terms of $S$? It is known that the answer is yes if $|S| = 3$ or if $S$ is a scaled subset of the square or triangular lattice (see \cite{BP} and \cite[Section~6.1]{BMP}). Brass and Pach state that it is likely that the answer is no for most other patterns. However, Theorem~\ref{thm:main} provides an affirmative answer to their problem for all patterns $S$ and also improves the bound by a polynomial factor in $n$.

A related problem is to understand how $f_S(n)$ varies with $S$ (see~\cite[(2)]{BP} for a general version of this problem). Pach conjectured~\cite[Section~6.1, Conjecture~1]{BMP} that there exists a triangle $T$ such that $\lim_{n \rightarrow \infty} f_T(n)/f_{\Delta}(n) = 0$, where $\Delta$ is a unit equilateral triangle. Theorem~\ref{thm:main} does not resolve this conjecture. However, the proof shows that the exponent obtained by our method depends on the arithmetic features of the points in $S$. It remains unclear whether this dependence reflects the true extremal behavior of $f_S(n)$ or only the limitations of the construction. In the former case, triangles of the form $T_z \colonequals \{0, 1, z\} \subset \C$ with $z$ transcendental could be plausible candidates to prove this conjecture. More generally, one may ask whether there exist finite nonempty sets $S, T \subset \R^2$ each containing at least two points such that $f_S(n) = o(f_T(n))$? In this direction, one can show that for each $c > 0$, there exist sets $S$ and $T$ such that $f_S(n) \le c f_T(n)$ for all $n$. This can be seen by taking $T = \{0,1\}$ and $S = \{0\} \cup \{z \colon z^m = 1\}$ with $m$ large.

\subsection{Proof overview}
\label{sec:overview}

We begin with a brief overview of the OpenAI unit distance construction. Since our proof builds upon Sawin's effective refinement~\cite{Saw}, we adopt his framing, but the core ideas are the same. The main arithmetic input is a tower of CM fields $K = F(i)$ of degree $2d$, a fractional ideal $I$ in $K$, and a totally positive element $\alpha \in F$ such that the set $U \colonequals \{\beta \in I \colon \beta \overline{\beta} = \alpha\}$ is large. This ideal $I$ maps to a lattice $\Lambda$ under the Minkowski embedding $\Phi \colon K \hookrightarrow \C^d$ which maps $x$ to $(\sigma_1(x), \dots, \sigma_d(x))$. Now consider the intersection of this lattice $\Lambda$ with a polydisc to obtain a set $X \subset \C^d$. One can show that this polydisc can be chosen so that for each $\beta \in U$, there are many $x \in X$ for which $x + \Phi(\beta)$ is also in $X$ (all $x \in X$ inside a slightly smaller polydisc have this property). Finally, projecting this set onto the complex plane via the map $\pi$ that sends $x$ to $x_1/\sqrt{\sigma_1(\alpha)}$ yields a set with many unit distances since
\begin{equation*}
    |\pi(x + \Phi(\beta)) - \pi(x)| = |\pi(\Phi(\beta))| = \left|\sigma_1(\beta)/\sqrt{\sigma_1(\alpha)}\right| = 1.
\end{equation*}

To accommodate an arbitrary pattern $S = \{z_1, \dots, z_k\}$, we express its points in a basis over the union $K_\infty$ of the fields in the tower. Specifically, one may choose a basis $\{\omega_1, \dots, \omega_r\}$ so that
\begin{equation*}
    z_i = \sum_{j=1}^r a_{ij} \omega_j,
    \qquad a_{ij}\in \mathcal{O}_{K_\infty},
\end{equation*}
for each $i \in [k]$. After discarding finitely many initial levels in the tower, all coefficients lie in the base field. We again consider the lattice $\Lambda \subset \C^d$ obtained from the ideal $I$ under the Minkowski embedding, but we now work with the product lattice $\Lambda^r \subset (\C^d)^r$. Again, we are able to choose a polydisc in $\C^{dr}$ such that its intersection $X$ with $\Lambda^r$ has the following property: for each $\beta \in U$, there are many $x \in X$ for which $x + (\Phi(a_{i1}\beta), \ldots, \Phi(a_{ir}\beta))$ is also in $X$ for all $i \in [k]$. We then project this set onto the complex plane via the map $\Pi$ that sends
\begin{equation*}
    (x_1,\ldots,x_r) \mapsto
    \sum_{j=1}^r \omega_j \frac{\pi_1(x_j)}{\sqrt{\sigma_1(\alpha)}},
\end{equation*}
where $\pi_1(x_j)$ denotes the first coordinate of $x_j$. The $K_\infty$-linear independence of the $\omega_j$'s guarantees that this projection is injective on $X$. Moreover, the projected point set contains many congruent copies of $S$ since for each $i \in [k]$, we have
\begin{equation*}
    \Pi(x + (\Phi(a_{i1}\beta), \ldots, \Phi(a_{ir}\beta))) = \Pi(x) + \sum_{j=1}^r \omega_j \frac{\sigma_1(a_{ij} \beta)}{\sqrt{\sigma_1(\alpha)}} = \Pi(x) + \frac{\beta}{\sqrt{\alpha}} \sum_{j=1}^r \omega_j a_{ij} = \Pi(x) + w_\beta z_i,
\end{equation*}
where $|w_\beta| = 1$. Here, we crucially take the embedding $\sigma_1$ to be the ordinary inclusion into $\C$.

Finally, since Sawin's number fields can be chosen to have degrees equal to all sufficiently large powers of $2$, we can assemble several far-apart translated blocks of the obtained configurations to obtain larger configurations with many congruent copies of $S$. This scaling argument yields the lower bound for all large enough $n$ with half the superlinear exponent.

\subsection{Organization}

In Section~\ref{sec:prelims}, we recall the number field construction from \cite{Saw}. In Section~\ref{sec:construction}, we describe our construction and prove Theorem~\ref{thm:main}. Finally, we conclude with some explicit examples and a discussion about the dependence of the exponent $\delta_S$ on $S$ in Section~\ref{sec:examples}.

\begin{note}
    The dependence of $\delta_S$ on the arithmetic features of $S$ in the present argument may not be intrinsic. We expect that a refinement of the construction should allow $\delta_S$ to be chosen to depend only on $|S|$. Since this refinement is not needed for the results proved here, we do not pursue it in the present version; we hope to include the details in a subsequent version.
\end{note}

\section{Preliminaries}
\label{sec:prelims}

In this section, we recall Sawin's number field construction~\cite{Saw}, which is the main number theoretic ingredient in our proof. We shall use the same explicit parameter choices as in his proof. To begin, recall the choice of sets of primes $T$ and $S_{\Q}$ from the proof of \cite[Theorem~1]{Saw}. Further, recall the choice of the function $k \colon S_{\Q} \rightarrow \N$ and write
\begin{align*}
    D &= \prod_{q \in T} q, \\
    Q_0 &= \Q(\sqrt{D}), \\
    \lambda &= \sqrt{4D}, \\
    Y &= \prod_{p \in S_{\Q}} p^{k(p)/(2e(p))}, \\
    \mu &= \frac{1}{2} \log\frac{2\pi}{e} + \sum_{p \in S_{\Q}} \frac{\log(k(p)+1)}{4e(p)} - \frac{1}{4} \log \lambda - \frac{1}{2} \log \log \lambda, \\
    C_0 &= 8 \lambda^2,
\end{align*}
where the function $e \colon S_{\Q} \rightarrow \N$ is defined as
\begin{equation*}
    e(p) = \begin{cases}
        2, & \text{ if } p = 2 \text{ or } p \in T,\\
        1, & \text{ otherwise}.
    \end{cases}
\end{equation*}
Further, writing $R_0 = 72$, one has
\begin{equation*}
    \frac{\mu + \log(1 - 1/R_0)}{\log (2 R_0 Y + 1)} \ge \delta_0.
\end{equation*}

We are now ready to state the arithmetic proposition which we shall use as a black box in our construction. It combines Lemmas~8, 9, 11, and 12, and Remark~13, along with the explicit choice of parameters in the proof of Theorem~1 in \cite{Saw}. We provide a short proof since Sawin~\cite{Saw} leaves some of the assertions in the proposition implicit.

\begin{proposition}
\label{prop:Sawin}
    There exists $\ell_0 \in \N$ and a nested sequence of totally real Galois extensions
    \begin{align*}
        F_0 \subset F_1 \subset F_2 \subset \dots
    \end{align*}
    of $\Q$ such that $d_m \colonequals [F_m \colon \Q] = 2^{m+\ell_0}$ and the fields $K_m \colonequals F_m(i)$ are nested CM fields. For every $m$, there is a fractional ideal $I_m$ of $K_m$ and a totally positive element $\alpha_m \in N_{K_m/F_m}(I_m)$, where $N_{K_m/F_m}(I_m)$ denotes the fractional ideal of $F_m$ generated by $\{\beta \overline{\beta} \colon \beta \in I_m\}$, for which
    \begin{align*}
        U_m \colonequals \{ \beta \in I_m \colon \beta \overline{\beta} = \alpha_m\} \quad \text{ and } \quad A_m \colonequals \left|\left(N_{K_m/F_m}(I_m)/(\alpha_m)\right)\right|
    \end{align*}
    satisfy $A_m^{1/(2d_m)} = Y$ and $C_0^{-1} \exp(2d_m\mu) \le |U_m| < \infty$. Moreover, $F_0$ contains $\Q(\{\sqrt{q} \colon q \in T\})$.
\end{proposition}

\begin{proof}
    Let $E = \Q\left(\{\sqrt{q} \colon q\in T\}\right)$. Let $\mathcal{L}/Q_0$ be the pro-$2$ extension whose Galois group is denoted by $G$ in \cite[Lemma~11]{Saw}. By \cite[Lemma~11]{Saw}, the group $G$ is infinite and $E\subseteq \mathcal{L}$. The conditions defining $\mathcal{L}$ are invariant under the nontrivial automorphism of $Q_0/\Q$. Hence, $\mathcal{L}/\Q$ is Galois. Thus, $\Gamma \colonequals \operatorname{Gal}(\mathcal{L}/\Q)$ is an infinite pro-$2$ group. Since $E/\Q$ is Galois, the subgroup $H_0 = \operatorname{Gal}(\mathcal{L}/E)$ is an infinite open normal subgroup of $\Gamma$.
    
    We construct recursively a descending sequence $H_0 \supset H_1 \supset H_2 \supset \dots$ of infinite open normal subgroups of $\Gamma$ such that $[H_m \colon H_{m+1}] = 2$. Suppose that $H_m$ has been constructed. Choose $x \in H_m \setminus \{1\}$. Then there is an open normal subgroup $W \trianglelefteq \Gamma$ with $x \notin W$. Let $V = W \cap H_m$, $P = \Gamma/V$, and $N = H_m/V$. Then $P$ is a finite $2$-group and $N$ is a nontrivial normal subgroup of $P$. Let $M$ be maximal among the proper subgroups of $N$ which are normal in $P$. The quotient $N/M$ is then a minimal nontrivial normal subgroup of $P/M$.
    
    Every nontrivial normal subgroup of a finite $2$-group meets its center nontrivially, as follows immediately from the class equation for the conjugation action. Hence, the minimality of $N/M$ implies that $N/M$ is central in $P/M$. It must therefore have order $2$, since otherwise it would contain a subgroup of order $2$, which would be a proper nontrivial normal subgroup of $P/M$. Taking $H_{m+1}$ to be the inverse image of $M$ gives the required subgroup.

    For $m \in \Z_{\ge 0}$, define $F_m = \mathcal{L}^{H_m}$. Then $F_0 = E$, $F_m \subset F_{m+1}$, $[F_{m+1} \colon F_m] = 2$, and every $F_m/\Q$ is Galois and totally real. Since the classes of the distinct primes in $T$ are linearly independent in $\Q^{\times}/(\Q^{\times})^2$, one has $[E \colon \Q] = 2^{|T|}.$ Consequently, $d_m=[F_m \colon \Q] = 2^{m+\ell_0}$, where $\ell_0 = |T|$. Moreover, $K_m = F_m(i)$ are nested CM fields.
    
    Note that each $F_m$ is a finite Galois subextension of $\mathcal{L}$ which contains $E$. Therefore the proof of \cite[Lemma~12]{Saw} applies to $F_m$ and $K_m$: the ramification indices and inertia degrees are the prescribed ones, every prime above $S_\Q$ splits in $K_m/F_m$, and $\operatorname{rd}_{K_m/F_m} = \lambda$. Applying \cite[Lemma~8]{Saw}, followed by the relative class-number bound \cite[Lemma~9]{Saw}, and substituting the explicit parameters from the proof of \cite[Theorem~1]{Saw}, gives a fractional ideal $I_m$ of $K_m$ and an element $\alpha_m \in N_{K_m/F_m}(I_m)$ satisfying $A_m^{1/(2d_m)} = Y$ and $|U_m| \ge C_0^{-1} \exp(2d_m \mu)$.
    
    The latter bound implies that $U_m$ is nonempty. Hence, for every real embedding $\tau \colon F_m \hookrightarrow \R$ and every extension $\sigma \colon K_m \hookrightarrow \C$, we have $\tau(\alpha_m) = |\sigma(\beta)|^2 > 0$ for $\beta \in U_m$. Therefore, $\alpha_m$ is totally positive. Finally, under the Minkowski embedding, $I_m$ is discrete, whereas the equations $|\sigma(\beta)|^2 = \sigma(\alpha_m)$ place the elements of $U_m$ in a compact product of circles. Thus, $U_m$ is finite.
\end{proof}

\section{Construction}
\label{sec:construction}

We begin our construction by fixing a tower of nested CM fields $K_0 \subset K_1 \subset K_2 \subset \dots$ coming from Proposition~\ref{prop:Sawin}. Define
\begin{equation*}
    K_\infty = \bigcup_{m=0}^\infty K_m.
\end{equation*}
Let $S \subset \C$ be a set consisting of $k \ge 2$ points. Suppose $S = \{z_1, \dots, z_k\}$. Let $V_\infty(S)$ denote the span of $S$ over $K_\infty$ and let $r = \mathrm{dim}_{K_\infty}(V_\infty(S))$. Let $B = \{\omega_1, \dots, \omega_r\}$ be any basis of $V_\infty(S)$. Then for each $i \in [k]$, we can find unique $(a_{ij})_{j \in [r]} \subset K_\infty$ such that 
\begin{equation*}
    z_i = \sum_{j=1}^r a_{ij} \omega_j.
\end{equation*}
Since the coefficient set $\{a_{ij} \colon i \in [k], j \in [r]\}$ is finite, all coefficients lie in $K_{m_0}$ for some integer $m_0 \ge 0$. We may assume without loss of generality that $a_{ij} \in \mathcal{O}_{K_{m_0}}$ for all $i \in [k]$ and $j \in [r]$ since otherwise the basis elements can be divided by appropriate positive integers to ensure this. Choose $C > 0$ such that
\begin{equation}
\label{eqn:C}
    \max_{\tau \colon K_{m_0} \hookrightarrow \C} \ \max_{i \in [k],\, j \in [r]} |\tau(a_{ij})| \le C.
\end{equation}

Now fix a level $m \ge m_0$ in the tower and write
\begin{equation*}
    F = F_m, \quad K = K_m, \quad d = d_m, \quad I = I_m, \quad U = U_m, \quad \alpha = \alpha_m.
\end{equation*}
Choose embeddings $\sigma_1, \dots, \sigma_d \colon K \hookrightarrow \C$, one from each conjugate pair, with $\sigma_1$ equal to the ordinary inclusion, to obtain the Minkowski embedding $\Phi \colon K \rightarrow \C^d$ defined by
\begin{equation*}
    \Phi(x) = (\sigma_1(x), \dots, \sigma_d(x))
\end{equation*}
for all $x \in K$. We equip $\C^d$ with the norm $\|\cdot\|_\alpha$ defined by
\begin{equation*}
    \|(w_1, \dots, w_d)\|_\alpha = \max_{\ell \in [d]} \frac{|w_\ell|}{\sqrt{\sigma_\ell(\alpha)}}.
\end{equation*}
This norm is well defined since $\alpha$ is totally positive (by Proposition~\ref{prop:Sawin}). Define
\begin{equation*}
    \Lambda \colonequals \Phi(I).
\end{equation*}
Then $\Lambda$ is a lattice in $\C^d$. As outlined in Section~\ref{sec:overview}, we shall consider the product lattice $\Lambda^r$ in $(\C^d)^r$, take a suitable finite subset of it, and project it onto the complex plane to obtain a set that contains many congruent copies of $S$.

For $R > 0$, let $B_\alpha(R)$ denote the closed ball of radius $R$ in $\C^d$ centered at the origin under the $\|\cdot\|_\alpha$ norm, namely,
\begin{equation*}
    B_\alpha(R) \colonequals \{w \in \C^d \colon \|w\|_\alpha \le R\}.
\end{equation*}
For each $i \in [k]$, we define a map $\varphi_i \colon U \rightarrow (\C^d)^r$ by
\begin{equation*}
    \varphi_i(\beta) = \left(\Phi(a_{i1}\beta), \ldots, \Phi(a_{ir}\beta)\right)
\end{equation*}
for all $\beta \in U$. Since $a_{ij} \in \mathcal O_K$ and $\beta \in I$, each coordinate of $\varphi_i(\beta)$ lies in $\Phi(I) = \Lambda$.
Let $R > 0$ be a constant to be chosen later. For $b =(b_1, \ldots, b_r) \in (\C^d/\Lambda)^r \cong \C^{dr}/\Lambda^r$, define
\begin{equation*}
    X_b = \prod_{j=1}^r \left((b_j+\Lambda) \cap B_\alpha(R)\right),  
\end{equation*}
and
\begin{equation*}
    E_b = \{(x,u) \in X_b \times U \colon x + \varphi_i(u) \in X_b \text{ for all } i \in [k]\}.
\end{equation*}

\begin{lemma}
\label{lem:exp}
    Suppose $R > C$. Then there exists $b \in (\C^d/\Lambda)^r$ such that
    \begin{equation*}
        |E_b| \ge  \left(1-\frac {C}{R}\right)^{2dr} |X_b| |U|.
    \end{equation*}
\end{lemma}

\begin{proof}
    Taking the average over $b \in (\C^d)^r/\Lambda^r$ with respect to the Haar probability measure, the unfolding identity gives
    \begin{equation}
    \label{eqn:avg}
        \E_b |X_b| = \frac{\operatorname{vol}(B_\alpha(R)^r)}{\operatorname{covol}(\Lambda^r)}.
    \end{equation}
    For any $i \in [k]$, $j \in [r]$, and $\beta \in U$, it follows from \eqref{eqn:C} that
    \begin{equation*}
        \|\Phi(a_{ij}\beta)\|_\alpha = \max_{\ell \in [d]} \frac{|\sigma_\ell(a_{ij}\beta)|}{\sqrt{\sigma_\ell(\alpha)}} = \max_{\ell \in [d]} |\sigma_\ell(a_{ij})|
        \le C,
    \end{equation*}
    since $\beta \overline{\beta} = \alpha$. Moreover, for each $b \in (\C^d/\Lambda)^r$ and $i \in [k]$, $x + \varphi_i(\beta) \in b + \Lambda^r$ since $\varphi_i(\beta) \in \Lambda^r$. Thus, for each fixed $\beta \in U$, the set of $x \in X_b$ for which $x + \varphi_i(\beta) \in X_b$ for all $i \in [k]$ contains $(b + \Lambda^r) \cap B_\alpha(R-C)^r$. Again, the unfolding identity gives
    \begin{equation*}
        \E_b |E_b| \ge \sum_{\beta \in U} \frac{\operatorname{vol}(B_\alpha(R-C)^r)}{\operatorname{covol}(\Lambda^r)} = \frac{\operatorname{vol}(B_\alpha(R-C)^r)}{\operatorname{covol}(\Lambda^r)} |U|.
    \end{equation*}
    Dividing the above inequality by \eqref{eqn:avg}, we obtain
    \begin{equation*}
        \E_b |E_b| \ge \frac{\operatorname{vol}(B_\alpha(R-C)^r)}{\operatorname{vol}(B_\alpha(R)^r)} |U| \cdot \E_b |X_b| = \left(1 - \frac{C}{R}\right)^{2dr} |U| \cdot \E_b |X_b|.
    \end{equation*}
    The equality holds since $\|\cdot\|_\alpha$ is a weighted infinity norm. Collecting all terms above on one side, we conclude that there exists some $b \in (\C^d/\Lambda)^r$ that satisfies the assertion of the lemma.
\end{proof}

Write $\pi_1 \colon \C^d \rightarrow \C$ for the first-coordinate projection. Define the projection $\Pi \colon (\C^d)^r \rightarrow \C$ by
\begin{equation*}
    \Pi(x_1,\ldots,x_r) = \sum_{j=1}^r \omega_j\frac{\pi_1(x_j)}{\sqrt{\sigma_1(\alpha)}}.
\end{equation*}
Let $P_b$ denote the planar point set obtained by projecting $X_b$ onto $\C$ via the above map $\Pi$. For a finite point set $P \subset \C$, let $f_S(P)$ denote the number of congruent copies of $S$ contained in $P$.

\begin{lemma}
\label{lem:count}
    For each $b \in (\C^d/\Lambda)^r$, we have $|P_b| = |X_b|$ and $f_S(P_b) \ge |E_b|/k!$.
\end{lemma}

\begin{proof}
    To show that $|P_b| = |X_b|$, it suffices to show that $\Pi$ is injective on $X_b$. Suppose $x, y \in X_b$ are such that $\Pi(x) = \Pi(y)$. Then for each $j \in [r]$, we have $x_j - y_j \in \Lambda$, and so
    \begin{equation*}
        x_j - y_j = \Phi(\xi_j)
    \end{equation*}
    for some $\xi_j \in I$. Now, the equality of projections gives
    \begin{equation*}
        \sum_{j=1}^r \omega_j \sigma_1(\xi_j) = 0.
    \end{equation*}
    Since $\sigma_1(K)=K\subset K_\infty$ and $\omega_1,\ldots,\omega_r$ are $K_\infty$-linearly independent, all $\sigma_1(\xi_j)$ vanish.  As $\sigma_1$ is an embedding, all $\xi_j=0$, and hence $x = y$. Thus $\Pi$ is injective on $X_b$.

    Let $(x,\beta) \in E_b$. Then for each $i \in [k]$, the point $x+\varphi_i(\beta)$ belongs to $X_b$, and
    \begin{align*}
        \Pi(x + \varphi_i(\beta)) &= \Pi(x) + \sum_{j=1}^r \omega_j \frac{\sigma_1(a_{ij}\beta)}{\sqrt{\sigma_1(\alpha)}} \\
        &= \Pi(x) + \frac{\beta}{\sqrt{\alpha}} \sum_{j=1}^r a_{ij} \omega_j \\
        &= \Pi(x) + \frac{\beta}{\sqrt{\alpha}} z_i.
    \end{align*}
    Here we used that $\sigma_1$ is the distinguished embedding, so it acts as the identity on $K \subset \C$. Thus
    \begin{equation*}
        \Pi(x) + w_\beta S \subseteq P_b,
    \end{equation*}
    where $w_\beta = \beta/\sqrt{\alpha}$. Since $|w_\beta| = 1$, this is a congruent copy of $S$ in $P_b$. Now note that the map from $E_b$ to $P_b^k$ that maps $(x, \beta)$ to
    \begin{equation*}
        \left(\Pi(x+\varphi_1(\beta)), \ldots, \Pi(x+\varphi_k(\beta))\right)
    \end{equation*}
    is injective. Given a $k$-tuple $(y_1, \dots, y_k)$ in the image of this map, one can recover $\beta$ by
    \begin{equation*}
        \beta = \sqrt{\alpha} \left( \frac{y_2 - y_1}{z_2 - z_1}\right).
    \end{equation*}
    Further, one may recover $\Pi(x)$ by
    \begin{equation*}
        \Pi(x) = y_1 - \frac{\beta}{\sqrt{\alpha}} z_1 = \frac{y_1 z_2 - y_2 z_1}{z_2 - z_1},
    \end{equation*}
    from which one can further recover $x$ since $\Pi$ is injective on $X_b$. Thus, each congruent copy of $S$ contained in $P_b$ can correspond to at most $k!$ pairs $(x, \beta) \in E_b$, from which the desired result follows immediately.
\end{proof}

\subsection{Construction for infinitely many $n$}

For $R > C$, define
\begin{equation}
\label{eqn:eta}
    \eta(R) \colonequals \mu + r \log\left(1-\frac {C}{R}\right).
\end{equation}
We now prove the desired superlinear lower bound along an infinite subsequence of cardinalities arising from the tower construction.

\begin{proposition}
\label{prop:subseq}
    Let $R > C$ be a constant such that $\eta(R) > 0$. Let $\delta_S = \eta(R)/(r\log(2RY+1))$ and $C_S = (C_0 |S|!)^{-1}$. Then $f_S(n) \ge C_S n^{1+\delta_S}$ for infinitely many $n$.
\end{proposition}

\begin{proof}
    Fix an integer $m \ge m_0$. Hereafter, we suppress $m$ from the notation. By Lemma~\ref{lem:exp}, we can find $b \in (\C^{d}/\Lambda)^r$ such that
    \begin{equation}
    \label{eqn:Eb}
        |E_b| \ge \left(1 - \frac{C}{R}\right)^{2dr} |X_{b}||U|.
    \end{equation}
    By Sawin's separation lemma~\cite[Lemma~4]{Saw}, every nonzero vector of $\Lambda = \Phi(I)$ has $\|\cdot\|_\alpha$-norm at least $Y^{-1}$. A standard packing argument similar to \cite[Lemma~2]{Saw} then gives
    \begin{equation}
    \label{eqn:Xb}
        |P_b| = |X_b| \le (2RY+1)^{2dr},
    \end{equation}
    where the equality follows from Lemma~\ref{lem:count}. Further, it follows from Lemma~\ref{lem:count} that
    \begin{equation*}
        f_S(P_b) \ge (|S|!)^{-1}|E_b|.
    \end{equation*}
    Using \eqref{eqn:Eb} followed by $|X_b| = |P_b|$ and the lower bound on $|U|$ from Proposition~\ref{prop:Sawin} gives
    \begin{equation}
    \label{eqn:fspb}
        f_S(P_b) \ge (|S|!)^{-1} \left(1 - \frac{C}{R}\right)^{2dr} |P_b| (C_0^{-1} e^{2 d \mu}) = C_S |P_b| e^{2d \eta(R)} \ge C_S |P_b|^{1+\delta_S},
    \end{equation}
    where $\delta_S$ and $C_S$ are as defined in the statement of the proposition. The last inequality above follows from \eqref{eqn:Xb}. Now, the argument in Lemma~\ref{lem:count} implies that the map from $E_b$ to $P_b^2$ that sends a pair $(x, \beta)$ to the ordered pair $\left(\Pi(x + \varphi_1(\beta)), \Pi(x + \varphi_2(\beta))\right)$ is injective. Hence,
    \begin{equation*}
        |E_b| \le |P_b|^2.
    \end{equation*}
    Using \eqref{eqn:Eb} followed by $|X_b| = |P_b|$ and the lower bound on $|U|$ from Proposition~\ref{prop:Sawin} gives
    \begin{equation*}
        n \colonequals |P_b| \ge C_0^{-1} \exp(2d\eta(R)).
    \end{equation*}
    Moreover, \eqref{eqn:fspb} implies $f_S(n) \ge C_S n^{1+\delta_S}$. Finally, as $m \rightarrow \infty$, Proposition~\ref{prop:Sawin} implies $d_m \rightarrow \infty$, which together with the above inequality implies $n \rightarrow \infty$, hence giving the desired result.
\end{proof}

\subsection{Construction for all large $n$}

Using the idea outlined in Section~\ref{sec:overview}, we provide a way to construct sets with $n$ points that contain many congruent copies of $S$, for all large enough $n$.

\begin{lemma}
\label{lem:large}
    $f_S(n) \ge (C_S/2) n^{1+\delta_S/2}$ for all sufficiently large $n$.
\end{lemma}

\begin{proof}
    Let $Z \colonequals (2RY+1)^{2r}$. The proof of Proposition~\ref{prop:subseq} gives, for $m \ge m_0$, a set $P_m$ with
    \begin{equation*}
        n_m \colonequals |P_m| \le Z^{d_m} \qquad \text{and} \qquad f_S(P_m)\ge C_S n_m \exp(2d_m\eta).
    \end{equation*}
    Recall from Proposition~\ref{prop:Sawin} that $d_m = 2^{m+\ell_0}$. Let $n \ge Z^{d_{m_0}}$, and choose $m$ maximal such that $Z^{d_m} \le n$. Then $n < Z^{d_{m+1}} = Z^{2d_m}$. Since $n_m \le Z^{d_m} \le n$, we take
    \begin{equation*}
        q \colonequals \left\lfloor\frac{n}{n_m}\right\rfloor
    \end{equation*}
    mutually far apart translated copies of $P_m$, and add $n-qn_m$ extra points far away to form a point set $P$. The internal copies of $S$ inside the translated blocks remain, so we obtain
    \begin{equation*}
        f_S(n) \ge q f_S(P_m) \ge \frac{n}{2 n_m} C_S n_m \exp(2 d_m \eta) \ge \frac{C_S}{2} n^{1+\eta/\log Z} = \frac{C_S}{2} n^{1+\delta_S/2}.
    \end{equation*}
    Since the above estimate holds for all $n$ large enough, this finishes the proof of the lemma.
\end{proof}

\subsection{The affinely independent case}

We now consider the generic case where the set $S$ is affinely independent over $\Qbar$. In the following lemma, we compute $\delta_S$ for such sets.

\begin{lemma}
\label{lem:affine}
    For sets $S$ that are affinely independent over $\Qbar$, one may take $\delta_S = \delta/(|S| \log |S|)$ for some absolute constant $\delta > 0$.
\end{lemma}

\begin{proof}
    Let $k = |S|$. Since $S$ is affinely independent over $\Qbar$, we may translate it to obtain $\{0,z_1,\ldots,z_{k-1}\}$, where $z_1, \ldots, z_{k-1}$ are linearly independent over $\Qbar$, and hence also over $K_\infty$. Hence, we may assume $S = \{0,z_1,\ldots,z_{k-1}\}$, where $z_1, \ldots, z_{k-1}$ are linearly independent over $K_\infty$. Thus, we may take the basis $B$ to be $S \setminus \{0\}$. Then the coefficients $(a_{ij})_{0 \le i < k,\, 1 \le j < k}$ are all $0$ or $1$, so we may take $C = 1$ and also $r = k-1$. Taking $R = rR_0$ in Proposition~\ref{prop:subseq} gives
    \begin{equation*}
        \delta_{S} = \frac{\mu + r \log \left(1-\frac{1}{rR_0}\right)}{
        r \log(2rR_0Y+1)} \ge \frac{\mu + \log(1 - 1/R_0)}{r \log(2rR_0Y+1)} \ge \frac{\delta}{|S| \log |S|}
    \end{equation*}
    for some absolute constant $\delta > 0$.
\end{proof}

\begin{proof}[Proof of Theorem~\ref{thm:main}]
    Recall the definition of $\eta(R)$ from \eqref{eqn:eta}. Note that $\eta(R) \rightarrow \mu$ as $R \rightarrow \infty$. Thus, one may choose $R$ large enough so that $\eta(R) > 0$. The main assertion of the theorem then follows from Proposition~\ref{prop:subseq} with such a choice of $R$. The two additional assertions follow from Lemmas~\ref{lem:large} and \ref{lem:affine}, in order.
\end{proof}

\section{Examples}
\label{sec:examples}

In this section, we compute the exponent $\delta_S$ for some specific choices of set $S$. Recall the constants $Y$, $\mu$, and $R_0$ from Section~\ref{sec:prelims}. Suppose $S = \{z_1, \dots, z_k\}$. Let $\{\omega_1, \dots, \omega_r\}$ be a basis of $\mathrm{Span}_{K_\infty}(S)$ such that $S$ admits a presentation with
\begin{equation*}
    z_i = \sum_{j=1}^r a_{ij} \omega_j
\end{equation*}
for all $i \in [k]$, where $\{a_{ij} \colon i \in [k], j \in [r]\} \subset \mathcal{O}_{K_\infty}$. Since the coefficient set is finite, it must be contained in $\mathcal{O}_{K_{m_0}}$ for some $m_0 \in \Z_{\ge 0}$. Suppose $C > 0$ is such that
\begin{equation*}
    \max_{\tau \colon K_{m_0} \hookrightarrow \C} \max_{i \in [k],\, j \in [r]} |\tau(a_{ij})| \le C.
\end{equation*}
Then for every $R > C$ satisfying
\begin{equation*}
    \mu + r \log \left(1-\frac{C}{R}\right) > 0,
\end{equation*}
the proof of Proposition~\ref{prop:subseq} gives the exponent
\begin{equation}
\label{eqn:delta}
    \delta_S = \frac{\mu + r \log \left(1-\frac{C}{R}\right)}{r \log(2RY+1)}.
\end{equation}
With the above explicit formula for $\delta_S$, we are ready to compute it for a few specific examples.

\begin{example}[The unit segment]
\label{ex:unit}
    Let $S_2 = \{0, 1\}$. Then $S_2$ has a one-dimensional presentation with basis $\{1\}$ and coefficients $0$ and $1$. Thus, $r = 1$. Further, we may take $C = 1$ since $\tau(0) = 0$ and $\tau(1) = 1$ for all embeddings $\tau \colon K_{m_0} \hookrightarrow \C$. Using Sawin's choice $R = R_0$ in \eqref{eqn:delta} gives
    \begin{equation*}
        \delta_{S_2} = \frac{\mu+\log(1-1/R_0)}{\log(2 R_0 Y+1)} \ge \delta_0.
    \end{equation*}
    Thus, the construction recovers Sawin's unit-distance exponent for $n$ in a subsequence of $\N$, and the block argument gives an exponent at least $\delta_0/2$ for all sufficiently large $n$.
\end{example}

\begin{example}[Cyclotomic triangles]
\label{ex:cyclotomic}
    For integer $k \ge 2$, let $T_k = \{0, 1, \zeta_k\}$, where $\zeta_k = e^{2 \pi i/k}$. If $\zeta_k \in K_\infty$, then $T_k$ has a one-dimensional presentation with basis $\{1\}$ and coefficients $0$, $1$, and $\zeta_k$. These are algebraic integers, and all their conjugates have modulus at most $1$, so $r = 1$ and $C = 1$. Hence, the same computation as in Example~\ref{ex:unit} gives $\delta_{T_k} \ge \delta_0$. Note that $\zeta_2 = -1$, $\zeta_4 = i$, and $\zeta_6 = (1 + i\sqrt{3})/2$ all lie inside $K_\infty$ since $i, \sqrt{3} \in E(i) \subseteq K_0 \subset K_\infty$. Therefore, $\delta_{T_k} \ge \delta_0$ for $k = 2$, $4$, and $6$. These correspond to a three-term arithmetic progression, a right-angled isosceles triangle, and an equilateral triangle, respectively.
    
    On the other hand, if $\zeta_k \notin K_\infty$, then $1$ and $\zeta_k$ are linearly independent over $K_\infty$. Thus, $T_k$ has a two-dimensional presentation with basis $\{1, \zeta_k\}$ and coefficients $0$ and $1$, so $r = 2$ and $C = 1$. Setting $R = 2 R_0$ in \eqref{eqn:delta} gives
    \begin{equation*}
        \delta_{T_k} \ge \frac{\mu + 2 \log \left(1-\frac{1}{2R_0}\right)}{2\log(4 R_0Y+1)} \ge \delta_1 \colonequals 0.007039.
    \end{equation*}
    Consequently, every cyclotomic triangle satisfies the uniform bound $\delta_{T_k} \ge \min \{\delta_0, \delta_1\} = \delta_1$, which is independent of $k$.
\end{example}

\begin{example}[Arithmetic progressions]
    For $k \in \N$, let $S_k = \{0, 1, \ldots, k\}$. Then $S_k$ has a one-dimensional presentation with basis $\{1\}$ and coefficients $0, 1, \ldots, k$. Hence, $r = 1$ and one may take $C = k$. Setting $R = kR_0$ in \eqref{eqn:delta} gives
    \begin{equation*}
        \delta_{S_k} \ge \frac{\mu + \log (1 - 1/R_0)}{\log(2k R_0 Y + 1)} \ge \frac{\delta}{\log (k+1)}
    \end{equation*}
    for some absolute constant $\delta > 0$. Note that the special cases $k = 1$ and $k = 2$ correspond to the unit segment and a three-term arithmetic progression, respectively. As computed in Examples~\ref{ex:unit} and \ref{ex:cyclotomic}, these cases enjoy the bound $\delta_{S_k} \ge \delta_0$.
\end{example}

\begin{example}[Regular polygons]
    For integer $k \ge 3$, let $R_k = \{1, \zeta_k, \ldots, \zeta_k^{k-1}\}$. If $\zeta_k \in K_\infty$, then $R_k$ has a one-dimensional presentation with basis $\{1\}$ and coefficients
    \begin{equation*}
        1, \zeta_k, \ldots, \zeta_k^{k-1}.
    \end{equation*}
    All conjugates of these coefficients have absolute value equal to $1$, so $r = 1$ and one may take $C = 1$. Hence, the same computation as in Example~\ref{ex:unit} gives $\delta_{R_k} \ge \delta_0$. The case $k = 3$ belongs to this category since $\zeta_3 = (-1+i\sqrt{3})/2 \in K_\infty$. It corresponds to an equilateral triangle, which we already studied in Example~\ref{ex:cyclotomic}.
    
    On the other hand, if $\zeta_k \notin K_\infty$, then letting $P(x)$ be the minimal polynomial of $\zeta_k$ over $K_\infty$ and $s = \mathrm{deg}(P)$, we have $2 \le s < k$ and $R_k$ has a presentation with basis $\{1, \zeta_k, \ldots, \zeta_k^{s-1}\}$. Suppose that $m_0 \in \Z_{\ge 0}$ is such that $K_{m_0}$ contains all the coefficients from the presentation. Then $K_{m_0}$ must contain the coefficients of $P(x)$ since they correspond to the coefficients that appear while expressing $\zeta_k^s$ as a linear combination of the basis elements. We will now show that the presentation coefficients are indeed algebraic integers. For $0 \le j < k$, let $Q_j(x)$ be the remainder of $x^j$ modulo $P(x)$, with $\deg(Q_j) < s$. Then $\zeta_k^j = Q_j(\zeta_k)$. Since $P$ divides $x^k-1$ over $K_{m_0}$, all roots of $P$ are $k$th roots of unity. Hence the coefficients of $P$ are algebraic integers in $K_{m_0}$. Since $P$ is monic, the remainders $Q_j(x)$ of $x^j$ modulo $P(x)$ all lie in $\mathcal O_{K_{m_0}}[x]$. Since the coefficients of $Q_j$ are precisely the coefficients in the presentation of $R_k$ with the chosen basis, the above argument implies that the presentation coefficients are all algebraic integers.
    
    Now let $\tau \colon K_{m_0} \hookrightarrow \C$ be any embedding. Then the roots of $\tau(P)$ are $s$ distinct $k$th roots of unity. Let $\Theta$ be the set of these roots. The polynomial $Q_j$ satisfies
    \begin{equation*}
        \tau(Q_j)(\theta) = \theta^j.
    \end{equation*}
    By Lagrange's interpolation formula, we obtain
    \begin{equation*}
        \tau(Q_j)(x) = \sum_{\theta \in \Theta}
        \theta^j \prod_{\substack{\eta\in\Theta\\ \eta\ne\theta}} \frac{x-\eta}{\theta-\eta}.
    \end{equation*}
    The numerator of each Lagrange basis polynomial has all coefficients bounded in absolute value by $2^{s-1}$. Further, for any $\theta \in \Theta$, we have
    \begin{equation*}
        \prod_{\substack{\eta\in\Theta\\ \eta\ne\theta}} |\theta-\eta| \ge \prod_{i=1}^{\left\lfloor \frac{s-1}{2} \right\rfloor}\frac{4i}{k} \prod_{i=1}^{\left\lceil \frac{s-1}{2} \right\rceil}\frac{4i}{k} \ge \left(\frac{2(s-1)}{ek}\right)^{s-1}.
    \end{equation*}
    Therefore, every coefficient of $\tau(Q_j)$ is
    bounded in absolute value by
    \begin{equation*}
        s\,2^{s-1} \left(\frac{ek}{2(s-1)}\right)^{s-1} = s \left(\frac{ek}{s-1}\right)^{s-1}.
    \end{equation*}
    Thus, $r = s$ and we may take $C$ equal to the above expression. Setting $R = s R_0 C$ in \eqref{eqn:delta} gives
    \begin{equation*}
        \delta_{R_k} \ge \frac{\mu + s \log \left(1 - \frac{1}{sR_0}\right)}{s\log(2s R_0 C Y+1)} \ge \frac{\mu + \log \left(1 - \frac{1}{R_0}\right)}{s\log(2s R_0 C Y+1)} \ge \frac{\delta'}{s^2 + s^2 \log (k/s)} \ge \frac{\delta}{k^2}
    \end{equation*}
    for some absolute constants $\delta, \delta' > 0$.
\end{example}

\section*{Acknowledgments}

The authors are grateful to Will Sawin for helpful comments. RG is supported by a joint Clarendon Fund and Exeter College SKP scholarship.

\bibliographystyle{amsplain}
\bibliography{bibliography}

\end{document}